\documentclass[a4paper,12pt]{amsart}
\usepackage[english]{babel}
\usepackage{amssymb,amsfonts,amsxtra,amsmath}
\usepackage{dsfont}
\renewcommand{\mathbb}{\mathds}
\usepackage{mathrsfs}
\DeclareMathAlphabet{\mathsc}{U}{rsfs}{m}{n}
\usepackage[colorlinks,plainpages,backref]{hyperref}
\usepackage[all,cmtip]{xy}
\usepackage{paralist}
\usepackage[shortlabels]{enumitem}
\usepackage{xspace}

\theoremstyle{plain}
\newtheorem{thm}{Theorem}[section]
\newtheorem{cor}[thm]{Corollary} 
\newtheorem{prp}[thm]{Proposition} 
\newtheorem{lem}[thm]{Lemma} 
\theoremstyle{definition}

\theoremstyle{remark} 
\newtheorem{rmk}[thm]{Remark}

\newtheorem{exa}[thm]{Example}

\newcommand{\mm}{\mathfrak{m}}

\renewcommand{\AA}{\mathbb{A}}
\newcommand{\CC}{\mathbb{C}}

\newcommand{\KK}{\mathbb{K}}
\newcommand{\NN}{\mathbb{N}}

\newcommand{\ZZ}{\mathbb{Z}}

\newcommand{\A}{\mathcal{A}}
\newcommand{\B}{\mathcal{B}}

\newcommand{\F}{\mathcal{F}}
\newcommand{\M}{\mathcal{M}}
\renewcommand{\O}{\mathcal{O}}

\newcommand{\an}{\mathrm{an}}

\newcommand{\ol}[1]{{\overline{#1}}}
\newcommand{\ul}[1]{{\underline{#1}}}
\newcommand{\oul}[1]{{\overline{\underline{#1}}}}
\newcommand{\wh}[1]{{\widehat{#1}}}
\newcommand{\wt}[1]{{\widetilde{#1}}}

\newcommand{\fs}[1]{{\left[\!\left[#1\right]\!\right]}}
\newcommand{\set}[1]{{\left\{#1\right\}}}
\newcommand{\ideal}[1]{{\left\langle#1\right\rangle}}

\newcommand{\onto}{\twoheadrightarrow}

\newcommand{\xmid}{\;\middle|\;}

\DeclareMathOperator{\gr}{gr}
\DeclareMathOperator{\height}{height}
\DeclareMathOperator{\Hom}{Hom}
\DeclareMathOperator{\Spec}{Spec}
\DeclareMathOperator{\id}{id}

\DeclareMathOperator{\ord}{ord}
\DeclareMathOperator{\inp}{inp}

\begin{document}
%%%%%%%%%%%%%%%%%%%%%%%%%%%%%%%%%%%%%%%%%%%%%%%%%%%%%%%%%%%%%%%%%%%%%%%%%%%%%%%%

\title[Set-theoretic complete intersection curves]{Deforming monomial space curves into set-theoretic complete intersection singularities}

\author[M.~Granger]{Michel Granger}
\address{
M.~Granger\\
Universit\'e d'Angers, D\'epartement de Math\'ematiques\\
LAREMA, CNRS UMR n\textsuperscript{o}6093\\
2 Bd Lavoisier\\
49045 Angers\\
France
}
\email{\href{mailto:granger@univ-angers.fr}{granger@univ-angers.fr}}
%\thanks{}

\author[M.~Schulze]{Mathias Schulze}
\address{M.~Schulze\\
Department of Mathematics\\
TU Kaiserslautern\\
67663 Kaiserslautern\\
Germany}
\email{\href{mailto:mschulze@mathematik.uni-kl.de}{mschulze@mathematik.uni-kl.de}}
%\thanks{}

%\date{\today}

% 32-xx Several complex variables and analytic spaces
% 32Sxx Singularities
% 32S30 Deformations of singularities; vanishing cycles
% 14-xx Algebraic geometry
% 14Hxx Curves
% 14H50 Plane and space curves
% 20-XX Group theory and generalizations
% 20Mxx Semigroups
% 20M25 Semigroup rings, multiplicative semigroups of rings


\begin{abstract}
We deform monomial space curves in order to construct examples of set-theoretical complete intersection space curve singularities.
As a by-product we describe an inverse to Herzog's construction of minimal generators of non-complete intersection numerical semigroups with three generators. 
\end{abstract}

\maketitle

%\tableofcontents

\numberwithin{equation}{section}

%%%%%%%%%%%%%%%%%%%%%%%%%%%%%%%%%%%%%%%%%%%%%%%%%%%%%%%%%%%%%%%%%%%%%%%%%%%%%%%%
\section{Introduction}
%%%%%%%%%%%%%%%%%%%%%%%%%%%%%%%%%%%%%%%%%%%%%%%%%%%%%%%%%%%%%%%%%%%%%%%%%%%%%%%%

It is a classical problem in algebraic geometry to determine the minimal number of equations that define a variety.
The codimension is a lower bound for this number which is reached in case of set-theoretic complete intersections.
Let $I$ be an ideal in a polynomial ring or a regular analytic algebra over a field $\KK$.
Then $I$ is called a set-theoretic complete intersection if $\sqrt{I}=\sqrt{I'}$ for some ideal $I'$ generated by $\height I$ many elements.
The subscheme or analytic subgerm $X$ defined by $I$ is also called a set-theoretic complete intersection in this case.
It is hard to determine whether a given $X$ is a set-theoretic complete intersection.
We address this problem in the case $I\in\Spec\KK\set{x,y,z}$ of irreducible analytic space curve singularities $X$ over an algebraically closed (complete non-discretely valued) field $\KK$.

Cowsik and Nori (see \cite{CN78}) showed that over a perfect field $\KK$ of positive characteristic any algebroid curve and, if $\KK$ is infinite, any affine curve is a set-theoretic complete intersection.
To our knowledge there is no example of an algebroid curve that is not a set-theoretic complete intersection.
Over an algebraically closed field $\KK$ of characteristic zero, Moh (see \cite{Moh82}) showed that an irreducible algebroid curve $\KK\fs{\xi,\eta,\zeta}\subset\KK\fs{t}$ is a set-theoretic complete intersection if the valuations $\ell,m,n=\upsilon(\xi),\upsilon(\eta),\upsilon(\zeta)$ satisfy
\begin{equation}\label{46}
\gcd(\ell,m)=1,\quad\ell<m,\quad(\ell-2)m<n.
\end{equation}

We deform monomial space curves in order to find new examples of set-theoretic complete intersection space curve singularities.
Our main result in Proposition~\ref{29} gives sufficient numerical conditions for the deformation to preserve both the value semigroup and the set-theoretic complete intersection property.
As a consequence we obtain 

%%%%%%%%%%%%%%%%%%%%%%%%%%%%%%%%%%%%%%%%%%%%%%%%%%%%%%%%%%%%%%%%%%%%%%%%%%%%%%%%

\begin{cor}\label{44}
Let $C$ be the irreducible curve germ defined by
\[
\O_C=\KK\set{t^\ell,t^m+t^p,t^n+t^q}\subset\KK\set{t}
\]
where $\gcd(\ell,m)=1$, $p>m$, $q>n$ and there are $a,b\ge2$ such that
\[
\ell=b+2,\quad m=2a+1,\quad n=ab+b+1.
\]
Let $\gamma$ be the conductor of the semigroup $\Gamma=\ideal{\ell,m,n}$ and set
\[
d_1=(a+1)(b+2),\quad\delta=\min\set{p-m,q-n}.
\] 
\begin{enumerate}[(a)]

\item\label{44a} If $d_1+\delta\geq\gamma$, then $\Gamma$ is the value semigroup of $C$.

\item\label{44b} If $d_1+\delta\geq\gamma+\ell$, then $C$ is a set-theoretic complete intersection.

\item\label{44c} If $a,b\ge3$ and $d_1+q-n\ge\gamma+\ell$, then $C$ defined by
\[
p:=\gamma-1-\ell>m
\]
is a non-monomial set-theoretic complete intersection.

\end{enumerate}
\end{cor}

%%%%%%%%%%%%%%%%%%%%%%%%%%%%%%%%%%%%%%%%%%%%%%%%%%%%%%%%%%%%%%%%%%%%%%%%%%%%%%%%

In the setup of Corollary~\ref{44} Moh's third condition in \eqref{46} becomes $ab<1$ and is trivially false.
Corollary~\ref{44} thus yields an infinite list of new examples of non-monomial set-theoretic complete intersection curve germs.

%%%%%%%%%%%%%%%%%%%%%%%%%%%%%%%%%%%%%%%%%%%%%%%%%%%%%%%%%%%%%%%%%%%%%%%%%%%%%%%%

Let us explain our approach and its context in more detail.
Let $\Gamma$ be a numerical semigroup.
Delorme (see \cite{Del76}) characterized the complete intersection property of $\Gamma$ by a recursive condition.
The complete intersection property holds equivalently for $\Gamma$ and its associated monomial curve $\Spec(\KK[\Gamma])$ (see \cite[Cor.~1.13]{Her70}) and is preserved under flat deformations.
For this reason we deform only non-complete intersection $\Gamma$.
A curve singularity inherits the complete intersection property from its value semigroup since it is a flat deformation of the corresponding monomial curve (see Proposition~\ref{20}).
The converse fails as shown by a counter-example of Herzog and Kunz (see \cite[p.~40-41]{HK71}).

In case $\Gamma=\ideal{\ell,m,n}$, Herzog (see \cite{Her70}) described minimal relations of the generators $\ell,m,n$.
There are two cases \eqref{H1} and \eqref{H2} (see \S\ref{31}) with $3$ and $2$ minimal relations respectively.
In the non-complete intersection case \eqref{H1} we describe an inverse to Herzog's construction (see Proposition~\ref{7}).
Bresinsky (see \cite{Bre79a}) showed (for arbitrary $\KK$) by an explicit calculation based on Herzog's case \eqref{H1} that any monomial space curve is a complete intersection.
Our results are obtained by lifting his equations to a (flat) deformation with constant value semigroup.
In section \S\ref{26} we construct such deformations (see Proposition~\ref{20}) following an approach using Rees algebras described by Teissier (see \cite[Appendix, Ch.~I, \S1]{Zar06}).
In \S\ref{42} we prove Proposition~\ref{29} by lifting Bresinsky's equations under the given numerical conditions.
In \S\ref{45} we derive Corollary~\ref{44} and give some explicit examples (see Example~\ref{50}).

It is worth mentioning that Bresinsky (see \cite{Bre79a}) showed (for arbitrary $\KK$) that all monomial Gorenstein curves in $4$-space are set-theoretic complete intersections.

%%%%%%%%%%%%%%%%%%%%%%%%%%%%%%%%%%%%%%%%%%%%%%%%%%%%%%%%%%%%%%%%%%%%%%%%%%%%%%%%
\section{Ideals of monomial space curves}\label{31}
%%%%%%%%%%%%%%%%%%%%%%%%%%%%%%%%%%%%%%%%%%%%%%%%%%%%%%%%%%%%%%%%%%%%%%%%%%%%%%%%

Let $\ell,m,n\in\NN$ generate a semigroup $\Gamma=\ideal{\ell,m,n}\subset\NN$.
\[
d=\gcd(\ell,m).
\]
We assume that $\Gamma$ is numerical, that is, $\gcd(\ell,m,n)=1$.

Let $\KK$ be a field and consider the map
\[
\varphi\colon\KK[x,y,z]\to\KK[t],\quad (x,y,z)\mapsto(t^\ell,t^m,t^n),
\]
whose image $\KK[\Gamma]=\KK[t^\ell,t^m,t^n]$ is the semigroup ring of $\Gamma$.

Pick $a,b,c\in\NN$ minimal such that
\[
a\ell=b_1m+c_2n,\quad bm=a_2\ell+c_1n,\quad cn=a_1\ell+b_2m
\]
for some $a_1,a_2,b_1,b_2,c_1,c_2\in\NN$.
Herzog distinguished two cases and proved the following statements (see \cite[Props.~3.3, 3.4, 3.5, Thm.~3.8]{Her70}).

\begin{enumerate}[label={(H\arabic*)}, ref=H\arabic*]

\item\label{H1} $0\notin\set{a_1,a_2,b_1,b_2,c_1,c_2}$.
Then
\begin{equation}\label{1}
a=a_1+a_2,\quad b=b_1+b_2,\quad c=c_1+c_2
\end{equation}
and the unique minimal relations of $\ell,m,n$ read
\begin{align}
\label{R1a}a\ell-b_1m-c_2n &= 0,\\
\label{R1b}-a_2\ell+bm-c_1n &= 0,\\
\label{R1c}-a_1\ell-b_2m+cn &= 0.
\end{align}
Their coefficients form the matrix
\begin{equation}\label{2}
\begin{pmatrix}
a & -b_1 & -c_2 \\
-a_2 & b & -c_1 \\
-a_1 & -b_2 & c
\end{pmatrix}.
\end{equation}
Accordingly the ideal $I=\ideal{f_1,f_2,f_3}$ of maximal minors
\begin{equation}\label{5}
f_1=x^a-y^{b_1}z^{c_2},\quad f_2=y^b-x^{a_2}z^{c_1},\quad f_3=x^{a_1}y^{b_2}-z^c
\end{equation}
of the matrix
\begin{equation}\label{23}
M_0=
\begin{pmatrix}
z^{c_1} & x^{a_1} & y^{b_1}  \\
y^{b_2} & z^{c_2} & x^{a_2}  \\
\end{pmatrix}.
\end{equation}
equals $\ker\varphi$, and the rows of this matrix generate the module of relations between $f_1,f_2,f_3$.
Here $\KK[\Gamma]$ is not a complete intersection.

\item\label{H2} $0\in\set{a_1,a_2,b_1,b_2,c_1,c_2}$. 
One of the relations $(a,-b,0)$, $(a,0,-c)$, or $(0,b,-c)$ is a minimal relation of $\ell,m,n$ and, up to a permutation of the variables, the minimal relations are
\begin{align}
\label{R2a}a\ell &= bm,\\
\label{R2b}a_1\ell+b_2m &= cn.
\end{align}
Their coefficients form the matrix
\begin{equation}\label{4}
\begin{pmatrix}
a & -b & 0 \\
-a_1 & -b_2 & c
\end{pmatrix}.
\end{equation}
It is unique up to adding multiples of the first row to the second.
Overall there are $3$ cases and an overlap case described equivalently by $3$ matrices
\begin{equation}\label{15}
\begin{pmatrix}
a & -b & 0 \\
a & 0 & c
\end{pmatrix},\quad
\begin{pmatrix}
a & -b & 0 \\
0 & -b & c
\end{pmatrix},\quad
\begin{pmatrix}
a & 0 & -c \\
0 & b & -c
\end{pmatrix}.
\end{equation}

Here $\KK[\Gamma]$ is a complete intersection.

\end{enumerate}

%%%%%%%%%%%%%%%%%%%%%%%%%%%%%%%%%%%%%%%%%%%%%%%%%%%%%%%%%%%%%%%%%%%%%%%%%%%%%%%%

In the following we describe the image of Herzog's construction and give a left inverse:
\begin{enumerate}[label={(H\arabic*')}, ref=H\arabic*']

\item\label{GS1} Given $a_1,a_2,b_1,b_2,c_1,c_2\in\NN\setminus\set{0}$, define $a,b,c$ by \eqref{1} and set
\begin{align}
\label{GS1a}\ell' &= b_1c_1+b_1c_2+b_2c_2 = b_1c+b_2c_2 = b_1c_1+bc_2,\\
\label{GS1b}m' &= a_1c_1+a_2c_1+a_2c_2 = ac_1+a_2c_2 = a_1c_1+a_2c,\\
\label{GS1c}n' &= a_1b_1+a_1b_2+a_2b_2 = a_1b+a_2b_2 = a_1b_1+ab_2,
\end{align}
and $e'=\gcd(\ell',m',n')$. 
Note that $\ell',m',n'$ are the submaximal minors of the matrix in \eqref{2}.

\item\label{GS2} Given $a,b,c\in\NN\setminus\set{0}$ and $a_1,b_2\in\NN$, define $\ell',m',n',d'$ by
\begin{align}
\label{GS2a}\ell' &= bd',\\
\label{GS2b}m' &= ad',\\
\label{GS2c}\frac {n'}{d'} &= \frac{a_1b+ab_2}c,\quad \gcd(n',d')=1.
\end{align}
\end{enumerate}

%%%%%%%%%%%%%%%%%%%%%%%%%%%%%%%%%%%%%%%%%%%%%%%%%%%%%%%%%%%%%%%%%%%%%%%%%%%%%%%%

\begin{rmk}
In the overlap case \eqref{15} the formulas \eqref{GS2a}-\eqref{GS2b} yield
\[
(\ell',m',n')=(bc,ac,ab).
\]
\end{rmk}

%%%%%%%%%%%%%%%%%%%%%%%%%%%%%%%%%%%%%%%%%%%%%%%%%%%%%%%%%%%%%%%%%%%%%%%%%%%%%%%%

\begin{lem}\label{3}
In case \eqref{H1}, let $\tilde n\in\NN$ be minimal with $x^{\tilde n}-z^{\tilde\ell}\in I$ for some $\tilde\ell\in\NN$.
Then $\gcd(\tilde\ell,\tilde n)=1$ and $(\tilde n,\tilde\ell)\cdot\gcd(b_1,b_2)=(n',\ell')$.
\end{lem}

\begin{proof}
The first statement holds due to minimality.
By Buchberger's criterion, the generators~\ref{5} form a Gr\"obner basis with respect to the reverse lexicographical ordering on $x,y,z$.
Let $g'$ denote a normal form of $g=x^{\tilde n}-z^{\tilde\ell}$ with respect to \ref{5}.
Then $g\in I$ if and only if $g'=0$.
By \eqref{1}, reductions by $f_2$ can be avoided in the calculation of $g$.
If $r_2$ and $r_1$ many reductions by $f_1$ and $f_3$ respectively are applied, then 
\[
g'=x^{\tilde n-a_1r_1-ar_2}y^{b_1r_2-r_1b_2}z^{r_1c+r_2c_2}-z^{\tilde\ell}
\]
and $g'=0$ is equivalent to
\[
\tilde\ell = r_1c+r_2c_2,\quad
b_1r_2 = r_1b_2,\quad
\tilde n = a_1r_1+ar_2.
\]
Then $r_i=\frac{b_i}{\gcd(b_1,b_2)}$ for $i=1,2$ and the claim follows.
\end{proof}

%%%%%%%%%%%%%%%%%%%%%%%%%%%%%%%%%%%%%%%%%%%%%%%%%%%%%%%%%%%%%%%%%%%%%%%%%%%%%%%%

\begin{lem}\label{6}\
\begin{asparaenum}[(a)]

\item\label{6a} In case \eqref{H1}, equations~\eqref{GS1a}-\eqref{GS1c} recover $\ell,m,n$.

\item\label{6b} In case \eqref{H2}, equations~\eqref{GS2a}-\eqref{GS2c} recover $\ell,m,n,d$.

\end{asparaenum}
\end{lem}

\begin{proof}\
\begin{asparaenum}[(a)]

\item Consider $\tilde n,\tilde\ell\in\NN$ as in Lemma~\ref{3}.
Then $x^{\tilde n}-z^{\tilde\ell}\in I=\ker\varphi$ means that $(t^{\ell})^{\tilde n}=(t^{n})^{\tilde\ell}$ and hence $\ell\tilde n=\tilde\ell n$.
So the pair $(\ell,n)$ is proportional to $(\tilde\ell,\tilde n)$ which in turn is proportional to $(\ell',n')$ by Lemma~\ref{3}.
Then the two triples $(\ell,m,n)$ and $(\ell',m',n')$ are proportional by symmetry. 
Since $\gcd(\ell,m,n)=1$ by hypothesis $(\ell',m',n')=q\cdot(\ell,m,n)$ for some $q\in\NN$. 
By Lemma~\ref{3}, $q$ divides $\gcd(b_1,b_2)$ and by symmetry also $\gcd(a_1,a_2)$ and $\gcd(c_1,c_2)$. 
By minimality of the relations~\eqref{R1a}-\eqref{R1c}, $\gcd(a_1,a_2,b_1,b_2,c_1,c_2)=1$ and hence $q=1$.
The claim follows.

\item By the minimal relation~\eqref{R2a}, $\gcd(a,b)=1$ and hence $(\ell,m)=d\cdot(b,a)$.
Substitution into equation~\eqref{R2b} and comparison with \eqref{GS2c} gives $\frac nd=\frac{a_1b+ab_2}c=\frac {n'}{d'}$ with $\gcd(n,d)=\gcd(\ell,m,n)=1$ by hypothesis.
We deduce that $(n,d)=(n',d')$ and then $(\ell,m)=(\ell',m')$.\qedhere

\end{asparaenum}
\end{proof}

%%%%%%%%%%%%%%%%%%%%%%%%%%%%%%%%%%%%%%%%%%%%%%%%%%%%%%%%%%%%%%%%%%%%%%%%%%%%%%%%

\begin{prp}\label{7}\
\begin{asparaenum}[(a)]

\item\label{7a} In case \eqref{GS1}, $a_1,a_2,b_1,b_2,c_1,c_2$ arise through \eqref{H1} from some numerical semigroup $\Gamma=\ideal{\ell,m,n}$ if and only if $e'=1$.
In this case, $(\ell,m,n)=(\ell',m',n')$.

\item\label{7b} In case \eqref{GS2}, $a,b,c,a_1,b_2$ arise through \eqref{H2} from some numerical semigroup $\Gamma=\ideal{\ell,m,n}$ if and only if $(\ell',m',n')$ is in the corresponding subcase of \eqref{H2},
\begin{gather}
\label{16a}\gcd(a,b)=1,\\
\label{16b}\forall q\in[-b_2/b,a_1/a]\cap\NN\colon\gcd(-a_1+qa,-b_2-qb,c)=1.
\end{gather}
In this case, $(\ell,m,n)=(\ell',m',n')$.

\end{asparaenum}
\end{prp}

\begin{proof}\
\begin{asparaenum}[(a)]

\item By Lemma~\ref{6}.\eqref{6a}, $e'=1$ is a necessary condition. 
Conversely let $e'=1$.
By definition, \eqref{2} is a matrix of relations of $(\ell',m',n')$.
Assume that $(\ell',m',n')$ is in case \eqref{H2}.
By symmetry, we may assume that $(\ell',m',n')$ admits a matrix of minimal relations
\begin{equation}\label{8}
\begin{pmatrix}
a' & -b' & 0 \\
-a'_1 & -b'_2 & c'
\end{pmatrix}
\end{equation}
of type \eqref{4}.
By choice of $a',b',c'$, it follows that
\[
a>a',\quad b>b',\quad c\geq c'.
\] 
By Lemma~\ref{6}.\eqref{6b}, $d'$ is the denominator of $\frac{a'_1b'+a'b'_2}{c'}$ and
\[
\ell'=b'd'.
\]
In particular $c'\geq d'$.
Then $b_1\geq b'$ contradicts \eqref{GS1a} since
\[
\ell'=b_1c+b_2c_2\geq b'c'+b_2c_2>b'c'\geq b'd'=\ell'.
\]
We may thus assume that $b_1<b'$.
The difference of first rows of \eqref{8} and \eqref{2} is then a relation
\[
\begin{pmatrix}a'-a & b_1-b' & c_2\end{pmatrix}
\]
of $(\ell',m',n')$ with $a'-a<0$, $b_1-b'<0$ and $c_2>0$.
Then $c_2\geq c'\ge d'$ by choice of $c'$.
This contradicts \eqref{GS1a} since
\[
\ell'=b_1c_1+bc_2\geq b_1c_1+b'd'>b'd'=\ell'.
\]
We may thus assume that $(\ell',m',n')$ is in case \eqref{H1} with a matrix of unique minimal relations 
\begin{equation}\label{9}
\begin{pmatrix}
a' & -b'_1 & -c'_2 \\
-a'_2 & b' & -c'_1 \\
-a'_1 & -b'_2 & c'
\end{pmatrix}
\end{equation}
of type \eqref{2} where 
\[
a'=a'_1+a'_2, \quad b'=b'_1+b'_2, \quad c'=c'_1+c'_2.
\]
as in \eqref{1}.
Then $(a,b,c)\ge(a',b',c')$ by choice of the latter and
\[
\ell'=b'_1c'+b'_2c'_2=b'_1c'_1+b'c'_2
\]
by Lemma~\ref{6}.\eqref{6a}.
If $(a_i,b_i,c_i)\geq(a'_i,b'_i,c'_i)$ for $i=1,2$, then
\[
\ell'=b_1c+b_2c_2\geq b'_1c'+b'_2c'_2=\ell'
\]
implies $c=c'$ and hence $(a,b,c)=(a',b',c')$ by symmetry.
By uniqueness of \eqref{9} then, $(a_1,a_2,b_1,b_2,c_1,c_2)=(a_1',a_2',b_1',b_2',c_1',c_2')$ and hence the claim.
By symmetry, it remains to exclude the case $c'_2>c_2$.
The difference of first rows of \eqref{9} and \eqref{2} is then a relation
\[
\begin{pmatrix}a'-a & b_1-b'_1 & c_2-c'_2\end{pmatrix}
\]
of $(\ell',m',n')$ with $a'-a\le0$, $c_2-c'_2<0$ and hence $b_1-b'_1\ge b'$ by choice of the latter.
This leads to the contradiction
\[
\ell'=b_2c_2+b_1c>b_1c\geq b'c'+b'_1c'>b_2'c_2'+b'_1c'=\ell'.
\]

\item By Lemma~\ref{6}.\eqref{6b}, the conditions are necessary.
Conversely assume that the conditions hold true.
By definition, \eqref{4} is a matrix of relations of $(\ell',m',n')$.
By hypothesis, \eqref{8} is a matrix of minimal relations of $(\ell',m',n')$.
By \eqref{16a}, $\gcd(\ell',m')=d'$ and hence by Lemma~\ref{6}.\eqref{6b}
\[
b=\frac{\ell'}{d'}=b',\quad a=\frac{m'}{d'}=a'.
\]
Writing the second row of \eqref{4} as a linear combination of \eqref{8} yields
\[
\begin{pmatrix}-a_1+qa & -b_2-qb & c\end{pmatrix}=
p\begin{pmatrix}-a'_1 & -b'_2 & c'\end{pmatrix}
\]
with $p\in\NN$ and $q\cap[-b_2/b,a_1/a]\cap\NN$ and hence $p=1$ by \eqref{16b}.
The claim follows.\qedhere

\end{asparaenum}
\end{proof}

%%%%%%%%%%%%%%%%%%%%%%%%%%%%%%%%%%%%%%%%%%%%%%%%%%%%%%%%%%%%%%%%%%%%%%%%%%%%%%%%

The following examples show some issues that prevent us from formulating stronger statement in Proposition~\ref{7}.\eqref{7b}.

\begin{exa}\
\begin{asparaenum}[(a)]

\item Take $(a,-b,0)=(3,-2,0)$ and $(-a_1,-b_2,c)=(-1,-4,4)$.
Then $(\ell',m',n')=(4,6,7)$ which is in case \eqref{H2}.
The second minimal relation is $(-2,-1,2)=\frac12((-a_1,-b_2,c)-(a,-b,0))$.
The same $(\ell',m',n')$ is obtained from $(a,0,-c)=(7,0,-4)$ and $(-a_2,b,-c_1)= (-1,3,-2)$.
This latter satisfies \eqref{16a} and \eqref{16b}, but $(a,0,-c)$ is not minimal.

\item Take $(a,-b,0)=(4,-3,0)$ and $(-a_1,-b_2,c)=(-2,-1,2)$.
Then $(\ell',m',n')=(3,4,5)$, but $(a,-b,0)$ is not a minimal relation.
In fact the corresponding complete intersection $\KK[\Gamma]$ defined by the ideal $\ideal{x^3-y^4,z^2-x^2y}$ is the union of two branches $x=t^3,y=t^4,z=\pm t^5$.

\end{asparaenum}
\end{exa}

%%%%%%%%%%%%%%%%%%%%%%%%%%%%%%%%%%%%%%%%%%%%%%%%%%%%%%%%%%%%%%%%%%%%%%%%%%%%%%%%
\section{Deformation with constant semigroup}\label{26}
%%%%%%%%%%%%%%%%%%%%%%%%%%%%%%%%%%%%%%%%%%%%%%%%%%%%%%%%%%%%%%%%%%%%%%%%%%%%%%%%

Let $\O=(\O,\mm)$ be a local $\KK$-algebra with $\O/\mm\cong\KK$.
Let $F_\bullet=\set{F_i\mid i\in\ZZ}$ be a decreasing filtration by ideals such that $F_i=\O$ for all $i\le0$ and $F_1\subset\mm$.
Consider the Rees ring
\[
A=\bigoplus_{i\in\ZZ}F_is^{-i}\subset\O[s^{\pm1}].
\]
It is a finite type graded $\O[s]$-algebra and flat (torsion free) $\KK[s]$-algebra with retraction
\[
A\onto A/A\cap\mm[s^{\pm1}]\cong\KK[s].
\]
For $u\in\O^*$ there are isomorphisms
\begin{equation}\label{24}
A/(s-u)A\cong\O,\quad A/sA\cong\gr^F\O.
\end{equation}
Geometrically $A$ defines a flat morphism with section
\[
\xymatrix{
\Spec(A)\ar[r]^-\pi & \AA^1_\KK\ar@/^1pc/[l]^-\iota
}
\]
with fibers over $\KK$-valued points 
\begin{align*}
\pi^{-1}(x)&\cong\Spec(\O),\quad\iota(x)=\mm,\quad 0\ne x\in\AA^1_\KK,\\
\pi^{-1}(0)&\cong\Spec(\gr^F\O),\quad\iota(0)=\gr^F\mm.
\end{align*}

%%%%%%%%%%%%%%%%%%%%%%%%%%%%%%%%%%%%%%%%%%%%%%%%%%%%%%%%%%%%%%%%%%%%%%%%%%%%%%%%

Let $\KK$ be an algebraically closed complete non-discretely valued field.
Let $C$ be an irreducible $\KK$-analytic curve germ.
Its ring $\O=\O_C$ is a one-dimensional $\KK$-analytic domain.
Denote by $\Gamma'$ its value semigroup.
Pick a representative $W$ such that $C=(W,w)$.
We allow to shrink $W$ suitably without explicit mention.
Let $\ol \O_W$ be the normalization of $\O_W$.
Then
\[
\xymatrix{
\ol\O_{W,w}=(\ol\O,\ol\mm)\cong(\KK\set{t'},\ideal{t'})\ar[r]^-\upsilon & \NN\cup\set\infty
}
\]
is a discrete valuation ring.
Denote by $\mm_W$ and $\ol\mm_W$ the ideal sheaves corresponding to $\mm$ and $\ol\mm$.
There are decreasing filtrations by ideal (sheaves)
\[
\F_\bullet=\ol\mm_W^\bullet\lhd\ol\O_W,\quad
F_\bullet=\F_{\bullet,w}=\ol\mm^\bullet=\upsilon^{-1}[\bullet,\infty]\lhd\ol\O.
\]
Setting $t=t'/s$ and identifying $\KK\cong\ol\O_W/\ol\mm_W$ this yields a finite extension of finite type graded $\O_W$- and flat (torsion free) $\KK[s]$-algebras
\begin{equation}\label{22}
\A=\bigoplus_{i\in\ZZ}(\F_i\cap\O_W)s^{-i}\subset\bigoplus_{i\in\ZZ}\F_is^{-i}=\ol\O_W[s,t]=\B\subset\ol\O_W[s^{\pm1}]
\end{equation}
with retraction defined by $\KK[s]\cong\B/(\B_{<0}+\B\ol\mm_W)$.
The stalk at $w$ is
\[
A=\A_w=\bigoplus_{i\in\ZZ}(F_i\cap\O)s^{-i}\subset\bigoplus_{i\in\ZZ}F_is^{-i}=\ol\O[s,t]=B\subset\ol\O[s^{\pm1}].
\]
At $w\ne w'\in W$ the filtration $\F_{w'}$ is trivial and the stalk becomes $\A_{w'}=\O_{W,w'}[s^\pm1]$.
The graded sheaves $\gr^\F\O_W\subset\gr^\F\ol\O_W$ are thus supported at $w$ and the isomorphism 
\[
\gr^\F(\ol\O_W)_w=\gr^F\ol\O\cong\KK[t']\cong\KK[\NN]
\]
identifies 
\begin{equation}\label{19}
(\gr^\F\O_W)_w=\gr^F\O\cong\KK[\Gamma'],\quad\Gamma'=\upsilon(\O\setminus\set{0})
\end{equation}
with the semigroup ring $\KK[\Gamma']$ of $\O$.

%%%%%%%%%%%%%%%%%%%%%%%%%%%%%%%%%%%%%%%%%%%%%%%%%%%%%%%%%%%%%%%%%%%%%%%%%%%%%%%%

The analytic spectrum $\Spec^\an_W(-)\to W$ applied to finite type $\O_W$-algebras represents the functor $T\mapsto\Hom_{\O_T}(-_T,\O_T)$ from $\KK$-analytic spaces over $W$ to sets (see \cite[Exp.~19]{Car62}).
Note that
\[
\Spec^\an_W(\KK[s])=\Spec^\an_\set{w}(\KK[s])=L
\]
is the $\KK$-analytic line.
The normalization of $W$ is 
\[
\nu\colon\ol W=\Spec^\an_W(\ol\O_W)\to W
\]
and $\B=\nu_*\ol\B$ where $\ol\B=\O_{\ol W}[s,t]$.
Applying $\Spec^\an_W$ to \eqref{22} yields a diagram of $\KK$-analytic spaces (see \cite[Appendix]{Zar06})
\begin{equation}\label{41}
\xymatrix{
X=\Spec^\an_W(\A)\ar[dr]^-\pi && \Spec^\an_W(\B)=Y\ar[ll]_-\rho\\
&L\ar[ur]^-\iota
}
\end{equation}
where $\pi$ is flat with $\pi\circ\rho\circ\iota=\id$ and
\begin{align*}
\pi^{-1}(x)&\cong\Spec^\an_W(\O_W)=W,\quad\iota(x)=w,\quad 0\ne x\in L,\\
\pi^{-1}(0)&\cong\Spec^\an_W(\gr^\F\O_W),\quad\iota(0)\leftrightarrow\gr^\F\mm_W.
\end{align*}

%%%%%%%%%%%%%%%%%%%%%%%%%%%%%%%%%%%%%%%%%%%%%%%%%%%%%%%%%%%%%%%%%%%%%%%%%%%%%%%%

\begin{rmk} 
Teissier defines $X$ as the analytic spectrum of $\A$ over $W\times L$ (see \cite[Appendix, Ch.~I, \S1]{Zar06}).
This requires to interpret the $\O_W$-algebra $\A$ as an $\O_{W\times L}$-algebra.
\end{rmk}

%%%%%%%%%%%%%%%%%%%%%%%%%%%%%%%%%%%%%%%%%%%%%%%%%%%%%%%%%%%%%%%%%%%%%%%%%%%%%%%%

\begin{rmk}
In order to describe \eqref{41} in explicit terms, embed
\[
\xymatrix{
L\supset\ol W\ar[r]^\nu & W\subset L^n
}
\]
with coordinates $t'$ and $\ul x=x_1,\dots,x_n$ and
\begin{align*}
X&=
\ol{\set{(\ul x,s)\xmid (s^{\ell_1}x_1,\dots,s^{\ell_n}x_n)\in W, s\ne 0}}\subset L^n\times L,\\
Y&=\set{(t,s)\xmid t'=st\in\ol W}\cup L\times\set{0}\subset L\times L.
\end{align*}
This yields the maps $X\to W\gets Y$.
The map $\rho$ in \eqref{41} becomes
\[
\rho(t,s)=(x_1(t')/s^{\ell_1},\dots,x_n(t')/s^{\ell_n})
\]
for $s\ne 0$ and the fiber $\pi^{-1}(0)$ is the image of the map 
\[
\rho(t,0)=((\xi_1(t),\dots,\xi_n(t)),0),\quad 
\xi_k(t)=\lim_{s\to 0}x_k(st)/s^{\ell_k}=\sigma(x_k)(t).
\] 
\end{rmk}

%%%%%%%%%%%%%%%%%%%%%%%%%%%%%%%%%%%%%%%%%%%%%%%%%%%%%%%%%%%%%%%%%%%%%%%%%%%%%%%%

Taking germs in \eqref{41} this yields the following.

%%%%%%%%%%%%%%%%%%%%%%%%%%%%%%%%%%%%%%%%%%%%%%%%%%%%%%%%%%%%%%%%%%%%%%%%%%%%%%%%
\begin{prp}\label{20}\pushQED{\qed}
There is a flat morphism with section
\[
\xymatrix{
S=(X,\iota(0))\ar[r]^-\pi & (L,0)\ar@/^1pc/[l]^-\iota
}
\]
with fibers 
\begin{align*}
\pi^{-1}(x)&\cong(W,w)=C,\quad\iota(x)=w,\quad 0\ne x\in L,\\
\pi^{-1}(0)&\cong\Spec^\an(\KK[\Gamma'])=C_0,\quad\iota(0)\leftrightarrow\KK[\Gamma'_+].\qedhere
\end{align*}
\end{prp}

%%%%%%%%%%%%%%%%%%%%%%%%%%%%%%%%%%%%%%%%%%%%%%%%%%%%%%%%%%%%%%%%%%%%%%%%%%%%%%%%

The structure morphism factorizes through a flat morphism
\[
\xymatrix{
X=\Spec^\an_W(\A)\ar@/_1pc/[rr]_-f\ar[r]^-{\hat f} & (|W|,\A)\ar[r] & W
}
\]
and $\hat f^\#_{\iota(0)}\colon A\to\O_{X,\iota(0)}$ induces an isomorphism of completions (see \cite[Exp.~19, \S2, Prop.~4]{Car62})
\[
\wh{A_{\iota(0)}}\cong\wh{\O_{X,\iota(0)}}.
\]
This yields the finite extension of $\KK$-analytic domains
\[
\O_S=\O_{X,\iota(0)}\subset\O_{Y,\iota(0)}.
\]
We aim to describe $\O_{Y,\iota(0)}$ and $\KK$-analytic algebra generators of $\O_S$. 
In explicit terms $\O_S$ is obtained from a presentation
\[
I\to\O[\ul x]\to A\to 0
\]
mapping $\ul x=x_1,\dots,x_n$ to $\iota(0)=A\cap\mm[s^{\pm1}]+As$ as
\begin{equation}\label{17}
\O_S=\O\set{\ul x}/\O\set{\ul x}I=\O\set{\ul x}\otimes_{\O[\ul x]}A,\quad\O\set{\ul x}=\O\wh\otimes\KK\set{\ul x}.
\end{equation}
Any $\O_W$-module $\M$ gives rise to an $\O_X$-module
\[
\wt\M=\O_X\otimes_{f^*\A}f^*\M=\hat f^*\M.
\]
With $M=\M_w$, its stalk at $\iota(0)$ becomes
\[
\wt M=\O_S\otimes_AM.
\]

%%%%%%%%%%%%%%%%%%%%%%%%%%%%%%%%%%%%%%%%%%%%%%%%%%%%%%%%%%%%%%%%%%%%%%%%%%%%%%%%

\begin{lem}\label{30}
$\Spec^\an_W(\B)=\Spec^\an_{\ol W}(\ol\B)$ and hence $\O_{Y,\iota(0)}=\KK\set{s,t}$.
\end{lem}

\begin{proof}
By finiteness of $\nu$ (see \cite[Exp.~19, \S3, Prop.~9]{Car62}),
\[
\ol\B=\wt{\nu_*\ol\B}=\wt\B=\O_{\ol W}\otimes_{\nu^*\ol\O_W}\nu^*\B.
\]
By the universal property of $\Spec^\an$, it follows that (see \cite[Thm.~2.2.5.(2)]{Con06})
\begin{align*}
\Spec^\an_{\ol W}(\ol\B)
&=\Spec^\an_{\ol W}(\O_{\ol W}\otimes_{\nu^*\ol\O_W}\nu^*\B)\\
&=\Spec^\an_{\ol W}(\O_{\ol W})\times_{\Spec^\an_{\ol W}(\nu^*\ol\O_W)}\Spec^\an_{\ol W}(\nu^*\B)\nonumber\\
&=\ol W\times_{\ol W\times_W\ol W}(\Spec^\an_W(\B)\times_W\ol W)\nonumber\\
&=\ol W\times_{\ol W}\Spec^\an_W(\B)\nonumber\\
&=\Spec^\an_W(\B).\qedhere\nonumber
\end{align*}
\end{proof}

%%%%%%%%%%%%%%%%%%%%%%%%%%%%%%%%%%%%%%%%%%%%%%%%%%%%%%%%%%%%%%%%%%%%%%%%%%%%%%%%

For $\xi'=\sum_{i\in\NN}\xi_i{t'}^i\in\KK[t']$ with $\ell=\upsilon(\xi')$ denote
\begin{equation}\label{33}
\xi=\xi'/s^\ell=\sum_{i\ge\ell}\xi_it^is^{i-\ell}\in F_\ell s^{-\ell}=B_\ell.
\end{equation}

%%%%%%%%%%%%%%%%%%%%%%%%%%%%%%%%%%%%%%%%%%%%%%%%%%%%%%%%%%%%%%%%%%%%%%%%%%%%%%%%

\begin{lem}\label{10}
Consider $\ul\xi'=\xi'_1,\dots,\xi'_n\in\mm\cap\KK[t']$, define $\ul\xi$ by \eqref{33} and $\ul\ell$ by $\ell_i=\upsilon(\xi'_i)$ for $i=1,\dots,n$.
If $\Gamma'=\ideal{\ul\ell}$, then $\O=\KK\set{\ul\xi'}$ and $\O_S=\KK\set{\ul\xi,s}$.
\end{lem}

\begin{proof}
By choice of $F_\bullet$, there is a cartesian square
\[
\xymatrix{
\llap{$B=\,$}\ol\O[t,s]\ar@{^(->}[r] & \ol\O[s^{\pm1}]\\
\llap{$A=\,$}\bigoplus_{i\in\ZZ}(F_i\cap\O)s^{-i}\ar@{^(->}[u]\ar@{^(->}[r] & \O[s^{\pm1}]\ar@{^(->}[u]
}
\]
of finite type graded $\O$-algebras.
Thus $\xi\in A\cap\mm[s^{\pm1}]$ if $\xi'\in\mm\cap k[t']$.

By hypothesis and \eqref{19}, the symbols $\sigma(\ul\xi')$ generate the graded $\KK$-algebra $\gr^F\O$.
Then $\ol{\sigma(\ul\xi')}=\sigma(\oul\xi')$ generate 
\[
\gr^F\mm/\gr^F\mm^2=\gr^F(\mm/\mm^2)
\]
and hence $\oul\xi'$ generate $\mm/\mm^2$ over $\KK$.
Then $\mm=\ideal{\ul\xi'}_\O$ by Nakayama's lemma and hence $\O=\KK\set{\ul\xi'}$ by the analytic inverse function theorem.

Under the graded isomorphism \eqref{24} with $\xi$ as in \eqref{33}
\[
\xymatrix@R=0em{
(A/As)_\ell\ar[r]^-{\cdot s^\ell} & \gr^F_\ell\O,\\
\ol\xi\ar@{|->}[r] & \sigma(\xi').
}
\]
The graded $\KK$-algebra $A/sA$ is thus generated by $\ol{\ul\xi}$.
Extend $F_\bullet$ to the graded filtration $F_\bullet[s^{\pm1}]$ on $\ol\O[s^{\pm1}]$.
For $i\ge j$,
\[
\xymatrix{
(A/As)_i=\gr^F_iA_i\ar[r]^-{\cdot s^{i-j}}_-\cong & \gr^F_iA_j.
}
\]
Thus finitely many monomials in $\ul\xi,s$ generate any $A_j/F_iA_j\cong F_j/F_i$ over $\KK$.
With $\gamma'$ the conductor of $\Gamma'$ and $i=\gamma'+j$, $F_{\gamma'}\subset\ol\mm\cap\O=\mm$ and hence $F_i=F_{\gamma'}F_j\subset\mm F_j$.
Therefore these monomials generate $A_j$ as $\O$-module by Nakayama's lemma.
It follows $A=\O[\ul\xi,s]$ as graded $\KK$-algebra.
Using $\O=\KK\set{\ul\xi'}$ and $\ul\xi'=\ul\xi s^{\ul\ell}$ then $\O_S=\KK\set{\ul\xi',\ul\xi,s}=\KK\set{\ul\xi,s}$ (see \eqref{17}).
\end{proof}

%%%%%%%%%%%%%%%%%%%%%%%%%%%%%%%%%%%%%%%%%%%%%%%%%%%%%%%%%%%%%%%%%%%%%%%%%%%%%%%%

We now reverse the above construction to deform generators of a semigroup ring.
Let $\Gamma$ be a numerical semigroup with conductor $\gamma$ generated by $\ul\ell=\ell_1,\dots,\ell_n$.
Pick corresponding indeterminates $\ul x=x_1,\dots,x_n$.
The weighted degree $\deg(-)$ defined by $\deg(\ul x)=\ul\ell$ makes $\KK[\ul x]$ a graded $\KK$-algebra and induces on $\KK\set{\ul x}$ a weighted order $\ord(-)$ and initial part $\inp(-)$ .
The assignment $x_i\mapsto\ell_i$ defines a presentation of the semigroup ring of $\Gamma$ (see \eqref{19})
\[
\KK[\ul x]/I\cong\KK[\Gamma]\subset\KK[t']\subset\KK\set{t'}=\ol\O.
\]
The defining ideal $I$ is generated by homogeneous binomials $\ul f=f_1,\dots,f_m$ of weighted degrees $\deg(\ul f)=\ul d$.
Consider elements $\ul\xi=\xi_1,\dots,\xi_n$ defined by
\begin{equation}\label{28}
\xi_j=t^{\ell_j}+\sum_{i\ge\ell_j+\Delta\ell_j}\xi_{j,i}t^is^{i-\ell_j}\in\KK[t,s]\subset\ol\O[t,s]=B
\end{equation}
with $\Delta\ell_j\in\NN\setminus\set{0}\cup\set{\infty}$.
Set
\[
\delta=\min\set{\Delta\ul\ell},\quad\Delta\ul\ell=\Delta\ell_1,\dots,\Delta\ell_n.
\]
With $\deg(t)=1=-\deg(s)$ $\ul\xi$ defines a map of graded $\KK$-algebras $\KK[\ul x,s]\to\KK[t,s]$ and a map of analytically graded $\KK$-analytic domains $\KK\set{\ul x,s}\to\KK\set{t,s}$ (see \cite{SW73} for analytic gradings).

%%%%%%%%%%%%%%%%%%%%%%%%%%%%%%%%%%%%%%%%%%%%%%%%%%%%%%%%%%%%%%%%%%%%%%%%%%%%%%%%

\begin{rmk}
Converse to \eqref{33}, any homogeneous $\xi\in\KK\set{t,s}$ of weighted degree $\ell$ can be written as $\xi=\xi'/s^\ell$ for some $\xi'\in\KK\set{t'}$.
It follows that $\xi(t,1)=\xi'(t)\in\KK\set{t}$.
\end{rmk}

%%%%%%%%%%%%%%%%%%%%%%%%%%%%%%%%%%%%%%%%%%%%%%%%%%%%%%%%%%%%%%%%%%%%%%%%%%%%%%%%

Consider the curve germ $C$ with $\KK$-analytic ring
\begin{equation}\label{39}
\O=\O_C=\KK\set{\ul\xi'},\quad\ul\xi'=\ul\xi(t,1),
\end{equation}
and value semigroup $\Gamma'\supset\Gamma$.

%%%%%%%%%%%%%%%%%%%%%%%%%%%%%%%%%%%%%%%%%%%%%%%%%%%%%%%%%%%%%%%%%%%%%%%%%%%%%%%%

We now describe when \eqref{28} generate the flat deformation in Proposition~\ref{20}.

%%%%%%%%%%%%%%%%%%%%%%%%%%%%%%%%%%%%%%%%%%%%%%%%%%%%%%%%%%%%%%%%%%%%%%%%%%%%%%%%

\begin{prp}\label{25}
The deformation \eqref{28} satisfies $\Gamma'=\Gamma$ if and only if there is a $\ul f'\in\KK\set{\ul x,s}^m$ with homogeneous components such that 
\begin{equation}\label{47}
\ul f(\ul\xi)=\ul f'(\ul\xi,s)s
\end{equation}
and $\ord(f'_i(\ul x,1))\ge d_i+\min\set{\Delta\ul\ell}$.
The flat deformation in Proposition~\ref{20} is then defined by
\begin{equation}\label{38}
\O_S=\KK\set{\ul\xi,s}=\KK\set{\ul x,s}/\ideal{\ul F},\quad\ul F=\ul f-\ul f's.
\end{equation}
\end{prp} 

\begin{proof}
First let $\Gamma'=\Gamma$.
Then Lemma~\ref{10} yields the first equality in \eqref{38}.
By flatness of $\pi$ in Proposition~\ref{20}, the relations $\ul f$ of $\ul\xi(t,0)=t^{\ul\ell}$ lift to relations $\ul F\in\KK\set{\ul x,s}^m$ of $\ul\xi$.
That is, $\ul F(\ul x,0)=\ul f$ and $\ul F(\ul\xi,s)=0$.
Since $\ul f$ and $\ul\xi$ have homogeneous components of weighted degrees $\ul d$ and $\ul\ell$, $\ul F$ can be written as $\ul F=\ul f-\ul f's$ where $\ul f'\in\KK\set{\ul x,s}^m$ has homogeneous components of weighted degrees $\ul d+\ul 1$.
This proves in particular the last claim.
Since $f_i(t^{\ul\ell})=0$, any term in $f'_i(\ul\xi,s)s=f_i(\ul\xi)$ involves a term of the tail of $\xi_j$ for some $j$.
Such a term is divisible by $t^{d_i+\Delta\ell_j}$ which yields the bound for $\ord(f'_i(\ul x,1))$.

Conversely let $\ul f'$ with homogeneous components satisfy \eqref{47}.
Suppose that there is a $k'\in\Gamma'\setminus\Gamma$.
Take $h\in\KK\set{\ul x}$ of maximal weighted order $k$ such that $\upsilon(h(\ul\xi'))=k'$.
In particular, $k<k'$ and $\inp h(t^{\ul\ell})=0$.
Then $\inp h\in I=\ideal{\ul f}$ and $\inp h=\sum_{i=1}^mq_if_i$ for some $\ul q\in\KK[\ul x]^m$.
Set 
\[
h'=h-\sum_{i=1}^mq_iF_i(\ul x,1)=h-\inp h+\sum_{i=1}^mq_if'_i(\ul x,1).
\]
Then $h'(\ul\xi')=h(\ul\xi')$ by \eqref{47} and hence $\upsilon(h'(\ul\xi'))=k'$.
With \eqref{47} and homogeneity of $\ul f'$ it follows that $\ord(h')>k$ contradicting the maximality of $k$.
\end{proof}

%%%%%%%%%%%%%%%%%%%%%%%%%%%%%%%%%%%%%%%%%%%%%%%%%%%%%%%%%%%%%%%%%%%%%%%%%%%%%%%%

\begin{rmk}
The proof of Proposition~\ref{25} shows in fact that the condition $\Gamma'=\Gamma$ is equivalent to the flatness of a homogeneous deformation of the parametrization as in \eqref{28}. 
These $\Gamma$-constant deformations are a particular case of $\delta$-constant deformations of germs of complex analytic curves (see \cite[\S3, Cor.~1]{Tei77}).
\end{rmk}

%%%%%%%%%%%%%%%%%%%%%%%%%%%%%%%%%%%%%%%%%%%%%%%%%%%%%%%%%%%%%%%%%%%%%%%%%%%%%%%%

The following numerical condition yields the hypothesis of Proposition~\ref{25}.

\begin{lem}\label{21}
If $\min\set{\ul d}+\delta\ge\gamma$ then $\Gamma'=\Gamma$.
\end{lem}

\begin{proof}
Any $k\in\Gamma'$ is of the form $k=\upsilon(p(\ul\xi'))$ for some $p\in\KK\set{\ul x}$ with $p_0=\inp(p)\in\KK[\ul x]$.
If $p_0(t^{\ul\ell})\ne 0$, then $k\in\Gamma$.
Otherwise, $p_0\in\ideal{\ul f}$ and hence $k\ge\min\set{\ul d}+\min\set{\ul \ell'}$.
The second claim follows.
\end{proof}

%%%%%%%%%%%%%%%%%%%%%%%%%%%%%%%%%%%%%%%%%%%%%%%%%%%%%%%%%%%%%%%%%%%%%%%%%%%%%%%%
\section{Set-theoretic complete intersections}\label{42}
%%%%%%%%%%%%%%%%%%%%%%%%%%%%%%%%%%%%%%%%%%%%%%%%%%%%%%%%%%%%%%%%%%%%%%%%%%%%%%%%

We return to the special case $\Gamma=\ideal{\ell,m,n}$ of \S\ref{31}.
Recall Bresinsky's method to show that $\Spec(\KK[\Gamma])$ is a set-theoretic complete intersection (see \cite{Bre79b}).
Starting from the defining equations \eqref{5} in case \eqref{H1} he computes
\begin{align*}
f_1^c=(x^a-y^{b_1}z^{c_2})^c
&=x^ag_1\pm y^{b_1c}z^{c_2c}\\
&=x^ag_1\pm y^{b_1c}z^{(c_2-1)c}(x^{a_1}y^{b_2}-f_3)\\
&=x^{a_1}g_2\mp y^{b_1c}z^{(c_2-1)c}f_3\\
&\equiv x^{a_1}g_2\mod\ideal{f_3}
\end{align*}
where $g_1\in\ideal{x,z}$ and
\[
g_2=x^{a-a_1}g_1\pm y^{b_1c+b_2}z^{(c_2-1)c}.
\]
He shows that if $c_2\ge2$, then further reducing $g_2$ by $f_3$ yields
\begin{align*}
g_2&=x^{a-a_1}g_1\pm y^{b_1c+b_2}z^{(c_2-2)c}(x^{a_1}y^{b_2}-f_3)\\
&\equiv x^{a-a_1}g_1\pm x^{a_1}y^{b_1c+2b_2}z^{(c_2-2)c}\mod\ideal{f_3}\\
&\equiv x^{a_1}\bigl(\tilde g_1+y^{b_1c+2b_2}z^{(c_2-2)c}\bigr)\mod\ideal{f_3}\\
&\equiv x^{a_1}g_3\mod\ideal{f_3}
\end{align*}
for some $\tilde g_1\in\KK[x,y,z]$.
Iterating $c_2$ many times yields a relation 
\begin{equation}\label{35}
f_1^c=qf_3+x^kg,\quad k=a_1c_2,
\end{equation}
where $g\equiv y^{\ell'}\mod\ideal{x,z}$ with $\ell'$ from \eqref{GS1a}.
One computes that
\[
x^{a_1}f_2=y^{b_1}f_3-z^{c_1}f_1,\quad
z^{c_2}f_2=x^{a_2}f_3-y^{b_2}f_1.
\]
Bresinsky concludes that
\begin{equation}\label{40}
Z(x,z)\not\subset Z(g,f_3)\subset Z(f_1,f_3)=Z(f_1,f_2,f_3)\cup Z(x,z)
\end{equation}
making $\Spec(\KK[\Gamma])=Z(g,f_3)$ a set-theoretic complete intersection.

%%%%%%%%%%%%%%%%%%%%%%%%%%%%%%%%%%%%%%%%%%%%%%%%%%%%%%%%%%%%%%%%%%%%%%%%%%%%%%%%

As a particular case of \eqref{28} consider three elements
\begin{align}\label{34}
\xi&=t^\ell+\sum_{i\ge\ell+\Delta\ell}\xi_is^{i-\ell}t^i,\\
\eta&=t^m+\sum_{i\ge m+\Delta m}\eta_is^{i-m}t^i,\nonumber\\
\zeta&=t^n+\sum_{i\ge n+\Delta n}\zeta_is^{i-n}t^i\in\KK[t,s].\nonumber
\end{align}
Consider the curve germ $C$ in \eqref{39} with $\KK$-analytic ring
\begin{equation}\label{27}
\O=\O_C=\KK\set{\xi',\eta',\zeta'},\quad (\xi',\eta',\zeta')=(\xi,\eta,\zeta)(t,1),
\end{equation}
and value semigroup $\Gamma'\supset\Gamma$.
We aim to describe situations where $C$ is a set-theoretic complete intersection under the hypothesis that $\Gamma'=\Gamma$.
By Proposition~\ref{25}, $(\xi,\eta,\zeta)$ then generate the flat deformation of $C_0=\Spec^\an(\KK[\Gamma])$ in Proposition~\ref{20}.
Let $F_1,F_2,F_3$ be the defining equations from Proposition~\ref{25}.

%%%%%%%%%%%%%%%%%%%%%%%%%%%%%%%%%%%%%%%%%%%%%%%%%%%%%%%%%%%%%%%%%%%%%%%%%%%%%%%%

\begin{lem}\label{36}
If $g$ in \eqref{35} deforms to $G\in\KK\set{x,y,z,s}$ such that
\begin{equation}\label{37}
F_1^c=qF_3+x^kG,\quad G(x,y,z,0)=g,
\end{equation}
then
\[
C=S\cap Z(s-1)=Z(G,F_3,s-1)
\]
is a set-theoretic complete intersection.
\end{lem}

\begin{proof}
Consider a matrix of indeterminates
\[
M=
\begin{pmatrix}
Z_1 & X_1 & Y_1\\
Y_2 & Z_2 & X_2\\
\end{pmatrix}
\]
and the system of equations defined by its maximal minors
\begin{align*}
F_1&=X_1X_2-Y_1Z_2,\\
F_2&=Y_1Y_2-X_2Z_1,\\
F_3&=X_1Y_2-Z_1Z_2.
\end{align*}
By Schaps' theorem (see \cite{Sch77}), there is a solution with coefficients in $\KK\set{x,y,z}\fs{s}$ that satisfies $M(x,y,z,0)=M_0$.
By Grauert's approximation theorem (see \cite{Gra72}), the coefficients can be taken in $\KK\set{x,y,z,s}$.
Using the fact that $M$ is a matrix of relations, we imitate in Bresinsky's argument in \eqref{40}, 
\[
Z(G,F_3)\subset Z(F_1,F_3)=Z(F_1,F_2,F_3)\cup Z(X_1,Z_2).
\]
The $\KK$-analytic germs $Z(G,F_3)$ and $Z(G,X_1,Z_2)$ are deformations of the complete intersections $Z(g,f_3)$ and $Z(g,x^{a_1},z^{c_2})$, and are thus of pure dimensions $2$ and $1$ respectively.
It follows that $Z(G,F_3)$ does not contain any component of $Z(X_1,Z_2)$ and must hence equal $Z(F_1,F_2,F_3)=S$.
The claim follows.
\end{proof}

%%%%%%%%%%%%%%%%%%%%%%%%%%%%%%%%%%%%%%%%%%%%%%%%%%%%%%%%%%%%%%%%%%%%%%%%%%%%%%%%

\begin{prp}\label{29}
Set $\delta=\min(\Delta\ell,\Delta m,\Delta n)$ and $k=a_1c_2$. 
Then the curve germ $C$ defined by \eqref{34} is a set-theoretic complete intersection if 
\begin{align*}
\min(d_1,d_2,d_3)+\delta&\geq\gamma,\\
\min(d_1,d_3)+\delta&\geq\gamma+k\ell,
\end{align*}
or, equivalently,
\begin{align*}
\min(d_1,d_2+k\ell,d_3)+\delta&\geq\gamma+k\ell.
\end{align*}
\end{prp}

\begin{proof}
By Lemma~\ref{21}, the first inequality yields the assumption $\Gamma'=\Gamma$ on \eqref{34}.
The conductor of $\xi^k\O$ equals $\gamma+k\ell$ and contains $(F_i-f_i)(\xi',\eta',\zeta')$, $i=1,3$, by the second inequality.
This makes $F_i-f_i$, $i=1,3$, divisible by $x^k$.
Substituting into \eqref{35} yields \eqref{37} and by Lemma~\ref{36} the claim.
\end{proof}

%%%%%%%%%%%%%%%%%%%%%%%%%%%%%%%%%%%%%%%%%%%%%%%%%%%%%%%%%%%%%%%%%%%%%%%%%%%%%%%%

\begin{rmk} 
We can permute the roles of the $f_i$ in Bresinsky's method. 
If the role of $(f_1,f_3)$ is played by $(f_1,f_2)$, we obtain a formula similar to \eqref{35}, $f_1^b=qf_2+x^kg$ with $k=a_2b_1$.
Instead of $x^k$, there is a power of $y$ if we use instead $(f_2,f_1)$ or $(f_2,f_3)$ and a power of $z$ if we use $(f_3,f_1)$ or $(f_3,f_1)$. 
The calculations are the same.
In the examples we favor powers of $x$ in order to minimize the conductor $\gamma+k\ell$.  
\end{rmk}

%%%%%%%%%%%%%%%%%%%%%%%%%%%%%%%%%%%%%%%%%%%%%%%%%%%%%%%%%%%%%%%%%%%%%%%%%%%%%%%%
\section{Series of examples}\label{45}
%%%%%%%%%%%%%%%%%%%%%%%%%%%%%%%%%%%%%%%%%%%%%%%%%%%%%%%%%%%%%%%%%%%%%%%%%%%%%%%%

Redefining $a,b$ suitably, we specialize to the case where the matrix in \eqref{23} is of the form
\[
M_0=
\begin{pmatrix}
z & x & y\\
y^b & z & x^a\\
\end{pmatrix}.
\]
By Proposition~\ref{7}.\eqref{7a}, these define $\Spec(\KK[\ideal{\ell,m,n}])$ if and only if 
\[
\ell=b+2,\quad m=2a+1,\quad n=ab+b+1(=(a+1)\ell-m),\quad\gcd(\ell,m)=1.
\]
We assume that $a,b\ge2$ and $b+2<2a+1$ so that $\ell<m<n$.
The maximal minors \eqref{5} of $M_0$ are then
\[
f_1=x^{a+1}-yz,\quad f_2=y^{b+1}-x^az,\quad f_3=z^2-xy^b
\]
with respective weighted degrees
\[
d_1=(a+1)(b+2),\quad d_2=(2a+1)(b+1),\quad d_3=2ab+2b+2
\]
where $d_1<d_3<d_2$.
In Bresinsky's method \eqref{35} with $k=1$ reads 
\[
f_1^2-y^2f_3=xg,\quad g=x^{2a+1}-2x^ayz+y^{b+2}.
\]

%%%%%%%%%%%%%%%%%%%%%%%%%%%%%%%%%%%%%%%%%%%%%%%%%%%%%%%%%%%%%%%%%%%%%%%%%%%%%%%%

We reduce the inequality in Proposition~\ref{29} to a condition on $d_1$.

\begin{lem}\label{43}
The conductor of $\xi\O$ is bounded by
\[
\gamma+\ell\leq d_2-\left\lfloor\frac{m}{\ell}\right\rfloor\ell<d_3.
\]
In particular, $d_2\geq\gamma+2\ell$ and $d_3>\gamma+\ell$.
\end{lem}

\begin{proof}
The subsemigroup $\Gamma_1=\ideal{\ell,m}\subset\Gamma$ has conductor
\[
\gamma_1=(\ell-1)(m-1)=2a(b+1)=n+(a-1)\ell+1\ge\gamma.
\]
To obtain a sharper upper bound for $\gamma$ we think of $\Gamma$ as obtained from $\Gamma_1$ by filling gaps of $\Gamma_1$.
Since $2n\ge\gamma_1$,
\[
\Gamma\setminus\Gamma_1=(n+\Gamma_1)\setminus\Gamma_1.
\]
The smallest elements of $\Gamma_1$ are $i\ell$ where $i=0,\dots,\left\lfloor\frac{m}{\ell}\right\rfloor$.
By symmetry of $\Gamma_1$ (see \cite{Kun70}), the largest elements of $\NN\setminus\Gamma_1$ are
\[
\gamma_1-1-i\ell=n+(a-1-i)\ell,\quad i=0,\dots,\left\lfloor\frac{m}{\ell}\right\rfloor,
\]
and contained in $n+\Gamma_1$ since the minimal coefficient $a-1-i$ is non-negative by 
\[
a-1-\left\lfloor\frac{m}{\ell}\right\rfloor\ge a-1-\frac{m}{\ell}=\frac{(a-1)b-3}{b+2}>-1.
\]
They are thus the largest elements of $\Gamma\setminus\Gamma_1$.
Their minimum attained at $i=\left\lfloor\frac{m}{\ell}\right\rfloor$ then bounds
\[
\gamma\leq\gamma_1-1-\left\lfloor\frac m{\ell}\right\rfloor\ell.
\] 
Substituting $\gamma_1+\ell-1=d_2$ yields the first particular inequality.
The second one follows from
\[
d_2-d_3=2a-b-1=m-\ell<\left\lfloor\frac{m}{\ell}\right\rfloor\ell.\qedhere
\]
\end{proof}

%%%%%%%%%%%%%%%%%%%%%%%%%%%%%%%%%%%%%%%%%%%%%%%%%%%%%%%%%%%%%%%%%%%%%%%%%%%%%%%%

\begin{proof}[Proof of Corollary~\ref{44}]\
\begin{asparaenum}[(a)]

\item This follows from Lemma~\ref{21}.

\item By Lemma~\ref{43}, the inequality in Proposition~\ref{29} simplifies to $d_1+\delta\ge\gamma+\ell$.
The claim follows.

\item Suppose that 
\[
d_1+q-n\ge\gamma+\ell
\]
for some $q>n$ and $a,b\ge3$.
Set $p=\gamma-1-\ell$.
Then $n>m+\ell$ and $\Gamma\cap(m+\ell,m+2\ell)$ can include at most $n$ and some multiple of $\ell$.
Since $\ell\ge 4$ it follows that $(m+\ell,m+2\ell)$ contains a gap of $\Gamma$ and hence $\gamma-1>\ell+m$ and $p>m$.
Moreover $(a-1)b\ge4$ is equivalent to 
\[
d_1+p-m\ge\gamma+\ell.
\]
By \ref{44b}, $C$ is a set-theoretic complete intersection.

It remains to show that $C\not\cong C_0$.
This follows from the fact that $\Omega^1_{C_0}\to\KK\set{t}dt$ has valuations $\Gamma\setminus\set{0}$ whereas the $1$-form
\[
\omega=mydx-\ell xdy=\ell(m-p)t^{p+\ell-1}dt\in\Omega^1_C\to\KK\set{t}dt
\]
has valuation $p+\ell=\gamma-1\not\in\Gamma$.\qedhere

\end{asparaenum}
\end{proof}

%%%%%%%%%%%%%%%%%%%%%%%%%%%%%%%%%%%%%%%%%%%%%%%%%%%%%%%%%%%%%%%%%%%%%%%%%%%%%%%%

\begin{exa}\label{50}
We discuss a list of special cases of Corollary~\ref{44}.

\begin{asparaenum}[(a)]

%%%%%%%%%%%%%%%%%%%%%%%%%%%%%%%%%%%%%%%%%%%%%%%%%%%%%%%%%%%%%%%%%%%%%%%%%%%%%%%%

\item\label{50a} $a=b=2$.
The monomial curve $C_0$ defined by $(x,y,z)=(t^4,t^5,t^7)$ has conductor $\gamma=7$.
Its only admissible deformation is
\[
(x,y,z)=(t^4,t^5+st^6,t^7).
\]
However, this deformation is trivial and our method does not yield a new example. 
To see this, we adapt a method of Zariski (see \cite[Ch.~III, (2.5), (2.6)]{Zar06}).
Consider the change of coordinates
\[
\tilde{x}=x+\frac{4s}{5}y=t^4+\frac{4s}{5}t^5+\frac{4s^2}5t^6
\]
and the change of parameters of the form $\tau=t+O(t^2)$ such that $\tilde{x}=\tau^4$. 
Then $\tau=t+\frac{s}{5}t^2+O(t^3)$ and hence $y=\tau^5+O(t^7)$ and $z=\tau^7+O(t^8)$. 
Since $O(t^7)$ lies in the conductor, it follows that $C\cong C_0$. 

In all other cases, Corollary~\ref{44} yields an infinite list of new examples.

%%%%%%%%%%%%%%%%%%%%%%%%%%%%%%%%%%%%%%%%%%%%%%%%%%%%%%%%%%%%%%%%%%%%%%%%%%%%%%%%

\item\label{50d} $a=3$, $b=2$.
Consider the monomial curve $C_0$ defined by $(x,y,z)=(t^4,t^7,t^9)$.
By Zariski's method from \eqref{50a}, we reduce to considering the deformation 
\[
(x,y,z)=(t^4,t^7,t^9+st^{10}).
\]
While part \ref{44c} of Corollary~\ref{44} does not apply, $C\not\cong C_0$ remains valid.
To see assume that $C_0\cong C$ induced by an automorphism $\varphi$ of $\CC\set{t}$.
Then $\varphi(x)\in\O_C$ shows that $\varphi$ has no quadratic term.
This, however, contradicts $\varphi(z)\in\O_C$.

%%%%%%%%%%%%%%%%%%%%%%%%%%%%%%%%%%%%%%%%%%%%%%%%%%%%%%%%%%%%%%%%%%%%%%%%%%%%%%%%

\item\label{50b} $a=b=3$.
The monomial curve $C_0$ defined by $(x,y,z)=(t^5,t^7,t^{13})$ has conductor $\gamma=17$.
We want to satisfy $p\ge\gamma+\ell-d_1+m=9$.
The most general deformation of $y$ thus reads
\[
y=t^7+s_1t^9+s_2t^{11}+s_3t^{16}.
\] 
The parameter $s_1$ can be again eliminated by Zariski's method as in \eqref{50a}.
This leaves us with the deformation 
\[
(x,y,z)=(t^5,t^7+s_2t^{11}+s_3t^{16},t^{13}+s_4t^{16})
\]
which is non-trivial due to part \ref{44c} of Corollary~\ref{44} with $p=11$.

%%%%%%%%%%%%%%%%%%%%%%%%%%%%%%%%%%%%%%%%%%%%%%%%%%%%%%%%%%%%%%%%%%%%%%%%%%%%%%%%

\item\label{50c} $a=8$, $b=3$.
The monomial curve $C_0$ defined by $(x,y,z)=(t^5,t^{17},t^{28})$ has conductor $\gamma=47$.
The condition in part~\ref{44b} of Corollary~\ref{44} requires $p\ge\gamma-d_1+m=19$.
In fact, the deformation
\[
(x,y,z)=(t^5,t^{17}+st^{18},t^{28})
\]
is not flat since $C$ has value semigroup $\Gamma'=\Gamma\cup\set{46}$.
However, $C$ is isomorphic to the general fiber of the flat deformation in $4$-space
\[
(x,y,z,w)=(t^5,t^{17}+st^{18},t^{28},t^{46}).
\]

%%%%%%%%%%%%%%%%%%%%%%%%%%%%%%%%%%%%%%%%%%%%%%%%%%%%%%%%%%%%%%%%%%%%%%%%%%%%%%%%

\end{asparaenum}
\end{exa}

%%%%%%%%%%%%%%%%%%%%%%%%%%%%%%%%%%%%%%%%%%%%%%%%%%%%%%%%%%%%%%%%%%%%%%%%%%%%%%%%
\bibliographystyle{amsalpha}
\bibliography{stcic}
%%%%%%%%%%%%%%%%%%%%%%%%%%%%%%%%%%%%%%%%%%%%%%%%%%%%%%%%%%%%%%%%%%%%%%%%%%%%%%%%
\end{document}